\documentclass[10pt,a4paper]{article}
\usepackage[dvips]{graphicx}
\usepackage{amssymb,amsmath,amsthm}
\usepackage{enumerate}
\usepackage{ascmac}
\usepackage[dvips]{graphicx}
\usepackage{float}
\usepackage{wrapfig}
\usepackage{fancybox}
\usepackage[usenames]{color}

\newtheorem{thm}{Theorem}[section]

\newtheorem{lem}{Lemma}[section]
\newtheorem{pro}{Proposition}[section]

\newcommand{\dis}{\displaystyle}

\newcommand{\R}{{\Bbb R}}
\newcommand{\N}{{\Bbb N}}

\newcommand{\pa}{\partial}

\newcommand{\hthree}{\hspace{3mm}}
\newcommand{\hfive}{\hspace{5mm}}

\def\cad#1{\csname #1\endcsname }
\usepackage{setspace} % setspaceパッケージのインクルード
\title{Blow-up sets for a complex valued semilinear heat equation}

\author{Junichi Harada
\\[1mm] {\small Faculty of Education and Human Studies, Akita University}}
\date{}

\begin{document}
\maketitle

%%%%%%%%%%%%%%%%%%%%%%%%%%%%%%%%%%%%%%%%%%%%%%%%%%%%%%%%%%%
\begin{abstract}
This paper is concerned with finite blow-up solutions of
a one dimensional complex-valued semilinear heat equation.
We provide locations and the number of blow-up points
from the viewpoint of zeros of the solution.
\end{abstract}
%%%%%%%%%%%%%%%%%%%%%%%%%%%%%%%%%%%%%%%%%%%%%%%%%%%%%%%%%%%

%%%%%%%%%%%%%%%%%%%%%%%%%%%%%%%%%%%%%%%%%%%%%%%%%%%%%%%%%%%%%%%%
\section{Introduction}
%%%%%%%%%%%%%%%%%%%%%%%%%%%%%%%%%%%%%%%%%%%%%%%%%%%%%%%%%%%%%%%%
We study blow-up solutions of a one dimensional complex-valued semilinear heat equation:
\begin{equation}\label{eq1A}
 z_t = z_{xx}+z^2,
\end{equation}
where $z(x,t)$ is a complex valued function and $x\in\R$.
If $z(x,t)$ is written by $z=a+ib$,
where $a,b\in\R$,
\eqref{eq1A} is rewritten as
\[
 a_t = a_{xx}+a^2-b^2,
\hfive
 b_t = b_{xx}+2ab.
\]
This equation is a special case of Constantin-Lax-Majda equation with a viscosity term,
which is a one dimensional model for the 3D Navier-Stokes equations
(see \cite{Constantin,Sakajo1,Sakajo2,Schochet,Yang,Guo}).
When $z$ is real-valued (i.e. $b\equiv0$),
\eqref{eq1A} coincides with the so-called Fujita equation \cite{Fujita}:
\begin{equation}\label{eq1B}
 a_t = a_{xx}+a^p.
\end{equation}
% This equation has been studied extensively by many authors (reference ***).
In a recent paper \cite{Guo},
they clarify the difference  the dynamics of solutions between \eqref{eq1A} and \eqref{eq1B}.
A goal of paper is to extend their results and
to provide new properties of solutions of \eqref{eq1A} based on results in \cite{Guo}.
The Cauchy problem of \eqref{eq1A} admits an unique local solution in
$L^\infty(\R)\cap C(\R)$.
We call a solution $z$ {\it blow-up} in a finite time,
if there exists $T>0$ such that
\[
 \limsup_{t\to T}\|z(t)\|_{L^{\infty}(\R)}
= \limsup_{t\to T}\sqrt{\|a(t)\|_{L^{\infty}(\R)}^2+\|b(t)\|_{L^{\infty}(\R)}^2}
= \infty.
\]
Moreover
we call a point $x_0\in\R$  a {\it blow-up point},
if there exists a sequence $\{(x_j,t_j)\}_{j\in\N}\subset\R\times(0,T)$ such that
$x_j\to x_0$, $t_j\to T$ and $|z(x_j,t_j)|\to\infty$ as $j\to\infty$.
The set of blow-up points is called a {\it blow-up set}.

We first consider an ODE solution $(a(x,t),b(x,t))=(a(t),b(t))$ of \eqref{eq1A}.
Then equation \eqref{eq1A} is reduced to
\[
 a_t = a^2-b^2,
\hfive
 b_t = 2ab.
\]
This ODE system has an unique solution given by
\[
 a(t) = \frac{T_1-t}{(T_1-t)^2+T_2^2},\hfive b(t) = \frac{T_2}{(T_1-t)^2+T_2^2},
\]
where 
\[
 T_1 = \frac{a(0)}{a(0)^2+b(0)^2},
\hfive
 T_2 = \frac{b(0)}{a(0)^2+b(0)^2}.
\]
Therefore
this ODE solution exists globally in time, if $b(0)\not=0$.
From this observation,
we expect that
the component $b$ prevents a blow-up phenomenon in \eqref{eq1A}.
In fact,
the following result is given in \cite{Guo}.
%%%%%%%%%%%%%%%%%%%%%%%%%%%%%%%%%%%%%%%%%%%%%%%%%%%%%%%%%%%%%%%%

%%%%%%%%%%%%%%%%%%%%%%%%%%%%%%%  Thm  %%%%%%%%%%%%%%%%%%%%%%%%%%%%%%%
\begin{thm}[{\rm Theorem 1.1 \cite{Guo}}]\label{11thm}
 Suppose that the initial data $(a_0,b_0)\in L^\infty(\R)\cap C(\R)$ satisfy
\[
 a_0(x)<Ab_0(x)
\hfive\mathrm{for\ all}\ x\in\R
\]
with some constant $A\in\R$.
Then the solution of \eqref{eq1A} exists globally in time and
$\dis\lim_{t\to\infty}(a(t),b(t))=(0,0)$ in $L^\infty(\R)$.
\end{thm}
%%%%%%%%%%%%%%%%%%%%%%%%%%%%%%%%%%%%%%%%%%%%%%%%%%%%%%%%%%%%%%%%
Furthermore for the case $b_0(x)>0$ with asymptotically positive constants,
they prove that the condition $a_0(x)<Ab_0(x)$ in Theorem \ref{11thm} is not needed
to assure the same conclusion.
%%%%%%%%%%%%%%%%%%%%%%%%%%%%%%%%%%%%%%%%%%%%%%%%%%%%%%%%%%%%%%%%

%%%%%%%%%%%%%%%%%%%%%%%%%%%%%%%  Thm  %%%%%%%%%%%%%%%%%%%%%%%%%%%%%%%
\begin{thm}[{\rm Theorem 1.4 \cite{Guo}}]\label{12thm}
 Suppose that the initial data $(a_0,b_0)\in L^\infty(\R)\cap C(\R)$ satisfy
\[
\begin{array}{cc}
 0\leq a_0\leq M, \hfive a_0\not\equiv M, \hfive 0\leq b_0\leq L
 \\[2mm] \dis
 \lim_{|x|\to\infty}a_0(x)=M, \hfive \lim_{|x|\to\infty}b_0(x)=N.
\end{array}
\]
for some $L>0$ and $M>N>0$.
Then the solution of \eqref{eq1A} exists globally in time and
$\dis\lim_{t\to\infty}(a(t),b(t))=(0,0)$ in $L^\infty(\R)$.
\end{thm}
%%%%%%%%%%%%%%%%%%%%%%%%%%%%%%%%%%%%%%%%%%%%%%%%%%%%%%%%%%%%%%%%
Our first result is a local version of Theorem \ref{12thm}.
To state our results,
we assume
\begin{equation}\label{eq1C}
 \sup_{0<t<T}(T-t)(\|a(t)\|_{L^\infty(\R)}+\|b(t)\|_{L^\infty(\R)})
<
 \infty.
\end{equation}
%%%%%%%%%%%%%%%%%%%%%%%%%%%%%%%%%%%%%%%%%%%%%%%%%%%%%%%%%%%%%%%%

%%%%%%%%%%%%%%%%%%%%%%%%%%%%%%%  Thm  %%%%%%%%%%%%%%%%%%%%%%%%%%%%%%%
\begin{thm}\label{thm2}
Let $(a,b)$ be a solution of \eqref{eq1A} and $T>0$ be its blow-up time.
Assume that \eqref{eq1C} holds and
there exists a neighborhood ${\cal O}$ of $(x_0,T)$ in $\R\times(0,T)$ such that
$b(x,t)>0$ or $b(x,t)<0$ for $(x,t)\in{\cal O}$.
Then $x_0$ is not a blow-up point of $(a,b)$.
\end{thm}
%%%%%%%%%%%%%%%%%%%%%%%%%%%%%%%%%%%%%%%%%%%%%%%%%%%%%%%%%%%%%%%%
Theorem \ref{thm2} implies that
if a solution $(a,b)$ blows up in a finite time,
the component $b$ must be sign changing near blow-up points.
A main goal of this paper is to characterize the location and the number of blow-up points
by using zeros of the component $b$.
%%%%%%%%%%%%%%%%%%%%%%%%%%%%%%%%%%%%%%%%%%%%%%%%%%%%%%%%%%%%%%%%

%%%%%%%%%%%%%%%%%%%%%%%%%%%%%%%  Thm  %%%%%%%%%%%%%%%%%%%%%%%%%%%%%%%
\begin{thm}\label{thm3}
Let $(a,b)$ and $T>0$ be as in {\rm Theorem \ref{thm2}}
and $\gamma(t)$ be a zero of $b(t)$ $(${\rm i.e.} $b(\gamma(t),t)=0)$.
Assume that \eqref{eq1C} holds and $b_0(x)$ has exact one zero.
Then
$\gamma(t)$ is continuous on $[0,T]$
and its blow-up point $x_0\in\R$ is given by $x_0=\gamma(T)$.
\end{thm}
%%%%%%%%%%%%%%%%%%%%%%%%%%%%%%%%%%%%%%%%%%%%%%%%%%%%%%%%%%%%%%%%
The existence of blow-up solutions are proved in \cite{Guo} and \cite{Nouaili}. 
In \cite{Guo},
they provide sufficient conditions on the initial data for a finite time blow-up
by using a comparison argument in the Fourier space based on \cite{Palais}.
In particular,
the exact initial data satisfying their blow-up conditions
is given by (see Remark 3.3 \cite{Guo})
\[
 a_0(x)=(3-4x^2)e^{-x^2},
\hfive
 b_0(x)=2xe^{-x^2}.
\]
For this case,
Theorem \ref{thm3} suggests that the solution blows up only on the origin.
On the other hand,
they \cite{Nouaili} construct blow-up solutions with exact blow-up profiles
$(a^*(x),b^*(x))=\lim_{t\to T}(a(x,t),b(x,t))$.
Two blow-up solutions constructed in \cite{Guo} and \cite{Nouaili}
have different type of asymptotic forms.
%%%%%%%%%%%%%%%%%%%%%%%%%%%%%%%%%%%%%%%%%%%%%%%%%%%%%%%%%%%%%%%%

%%%%%%%%%%%%%%%%%%%%%%%%%%%%%%%%%%%%%%%%%%%%%%%%%%%%%%%%%%%%%%%%
\section{Preliminary}
\subsection{Functional setting}\label{Sub21}
To study the asymptotic behavior of blow-up solutions,
we introduce a self-similar variable around $\xi\in\R$
and a new unknown function $(u_\xi,v_\xi)$,
which is defied by
\begin{equation}\label{eq2A}
 u_\xi(y,s)=(T-t)a(\xi+e^{-s/2}y,t),
\hfive
 v_\xi(y,s)=(T-t)b(\xi+e^{-s/2}y,t),
\hfive
 t=T-e^{-s}.
\end{equation}
Let $s_T=-\log(T-t)$.
Then $(u,v)=(u_\xi,v_\xi)$ satisfies
\begin{equation}\label{eq2B}
 \begin{cases}
 \dis
 u_s = u_{yy}-\frac{y}{2}u_y-u+u^2-v^2, & y\in\R,\ s>s_T,
 \\[3mm] \dis
 v_s = v_{yy}-\frac{y}{2}v_y-v+2uv, & y\in\R,\ s>s_T.
 \end{cases}
\end{equation}
We here introduce functional spaces which are related to \eqref{eq2B}.
Put $\rho(y)=e^{-y^2/4}$ and
\[
 L_\rho^2(\R) = \left\{ f\in L_{\text{loc}}^2(\R); \|f\|_\rho<\infty\right\},
\hfive
 H_\rho^1(\R) =
 \left\{
 f\in L_\rho^2(\R);\|f\|_{H_\rho^1(\R)}=\sqrt{\|f\|_\rho^2+\|f_x\|_\rho^2}<\infty
 \right\},
\]
where the norm of $L_\rho^2(\R)$ is defined by
\[
 \|f\|_\rho^2 = \int_{-\infty}^\infty f(y)^2\rho(y)dy.
\]
Here we recall the following inequality (see Lemma 2.1 \cite{Naito} p. 430).
\begin{equation}\label{5FFGeq}
 \int_{-\infty}^\infty y^2v^2\rho dy < c\|f\|_{H_\rho^1(\R)}^2.
\end{equation}
For the convenience of the reader,
we provide the proof of this inequality.
Let $g(y)=f(y)e^{-y^2/8}$.
Then a direct computation shows that
\[
 g_y^2 = \left( f_y^2+\frac{y^2}{16}f^2-\frac{y}{4}(f^2)_y \right)e^{-y^2/4}
\]
The integration of the last term is calculated as
\[
 -\int_{-\infty}^\infty
 \frac{y}{4}(f^2)_ye^{-y^2/4}dy
=
 \int_{-\infty}^\infty
 \left( \frac{y}{4}e^{-y^2/4} \right)_yf^2dy
=
 \int_{-\infty}^\infty
 \left( \frac{1}{4}-\frac{y^2}{8} \right)f^2e^{-y^2/4}dy.
\]
Therefore
we obtain
\[
 \int_{-\infty}^\infty f_y^2\rho dy+\frac{1}{4}\int_{-\infty}^\infty f^2\rho dy
 -\frac{1}{16}\int_{-\infty}^\infty y^2f^2\rho dy >0,
\]
which proves the desired inequality.
%%%%%%%%%%%%%%%%%%%%%%%%%%%%%%%%%%%%%%%%%%%%%%%%%%%%%%%%%%%%%%%%

%%%%%%%%%%%%%%%%%%%%%%%%%%%%%%%%%%%%%%%%%%%%%%%%%%%%%%%%%%%%%%%%
\subsection{Boundedness of solutions in self-similar variables}
%%%%%%%%%%%%%%%%%%%%%%%%%%%%%%%%%%%%%%%%%%%%%%%%%%%%%%%%%%%%%%%%
We here provide some conditions for a boundedness of solutions.
These conditions are useful to apply a scaling argument,
which is often used in the proof of Theorem \ref{thm2} and Theorem \ref{thm3}.
%%%%%%%%%%%%%%%%%%%%%%%%%%%%%%%%%%%%%%%%%%%%%%%%%%%%%%%%%%%%%%%%

%%%%%%%%%%%%%%%%%%%%%%%%%%%%%%%  Lem  %%%%%%%%%%%%%%%%%%%%%%%%%%%%%%%
\begin{lem}\label{2ALem}
Let $(a,b)$ be a solution of \eqref{eq1A} satisfying \eqref{eq1C}
and $(u_\xi,v_\xi)$ be given in \eqref{eq2A}.
Then there exist $R>0$ and $\epsilon_0>0$ such that
if $\|u_\xi(s_1)\|_{L^\infty(-R,R)}+\|v_\xi(s_1)\|_{L^\infty(-R,R)}<\epsilon_0$
for some $\xi\in\R$ and $s_1>s_T$,
then $\xi$ is not a blow-up point of $(a,b)$.
\end{lem}
%%%%%%%%%%%%%%%%%%%%%%%%%%%%%%%  Proof  %%%%%%%%%%%%%%%%%%%%%%%%%%%%%%%
\begin{proof}
For simplicity of notations,
we omit the subscript $\xi\in\R$.
Let $M=\sup_{s>0}(\|u(s)\|_{L^\infty(\R)}+\|u(s)\|_{L^\infty(\R)})<\infty$
and set
$F(s)=\|u(s)\|_\rho^2+\|v(s)\|_\rho^2$, $G(s)=\|u_y(s)\|_\rho^2+\|v_y(s)\|_\rho^2$.
Multiplying \eqref{eq2B} by $u$ and $v$,
we get
\[
 \frac{1}{2}F_s
<
 -G-F+c\int_{-\infty}^\infty\left( |u|^3+|u|^3 \right)\rho dy.
\]
We assume $\|u(s)\|_{L^\infty(-R,R)}+\|v(s)\|_{L^\infty(-R,R)}<\epsilon$.
Then from \eqref{eq2B},
the last term is estimated by
\[
 \int_{-\infty}^\infty\left( |u|^3+|v|^3 \right)\rho dy
<
 \epsilon\int_{|y|<R}\left( u^2+v^2 \right)\rho dy
+
 MR^{-2}\int_{|y|>R}y^2\left( u^2+v^2 \right)\rho dy
<
 \epsilon F
+
 cMR^{-2}G.
\]
We now choose $R_0>0$ and $\epsilon_0>0$ such that
$\epsilon_0<1/2$ and $cMR_0^{-2}<1/2$,
which implies $F_s(s)<0$
if $\|u(s)\|_{L^\infty(-R_0,R_0)}+\|v(s)\|_{L^\infty(-R_0,R_0)}<\epsilon_0$.
To construct a comparison function for $v$,
we first consider
\[
 w_s = w_{yy}-\frac{y}{2}w_y+(-1+2M)w \hthree \tau>s,
\hfive
 w(s) = |v(s)|.
\]
We easily see that
\[
 \|w(\tau)\|_\rho^2<e^{(-2+4M)(\tau-s)}\|v(s)\|_\rho^2.
\]
Next we construct a comparison function for $u$.
\[
 z_s = z_{yy}-\frac{y}{2}z_y+(-1+M)z+w(\tau)^2 \hthree \tau>s,
\hfive z(s)=|u(s)|,
\]
where $w(\tau)$ is defined above.
Then we get
\[
\begin{array}{lll}
 \|z(\tau)\|_\rho^2
&\hspace{-2mm}<&\hspace{-2mm} \dis
 e^{(-1+2M)(\tau-s)}\|u(s)\|_\rho^2+M^2\int_s^\tau e^{(-1+2M)(\tau-\mu)}\|w(\mu)\|_\rho^2d\mu
\\[4mm]
&\hspace{-2mm}<&\hspace{-2mm} \dis
 e^{(-1+2M)(\tau-s)}\|u(s)\|_\rho^2+
 \left( \frac{M^2}{-2+4M} \right)
 e^{(-3+6M)(\tau-s)}\|v(s)\|_\rho^2.
\end{array}
\]
Combining above estimates,
we obtain
\begin{equation}\label{eq2C}
 F(\tau) < c_1e^{c_2(\tau-s)}F(s)
\hfive\text{for } \tau>s
\end{equation}
for some $c_1,c_2>0$.
Furthermore
by a regularity theory for parabolic equations,
it holds that
\begin{equation}\label{eq2D}
 \|u(s)\|_{L^\infty(-R_0,R_0)}+\|v(s)\|_{L^\infty(-R_0,R_0)}
<
 c_3\int_{s-1}^sF(\mu) d\mu.
\end{equation}
Let $\epsilon_1=\min\{c_1e^{c_2}/2,c_3/2\}\epsilon_0$ and $\epsilon_2=\epsilon_1/2$.
We now claim that if $F(s)<\epsilon_2$ for some $s>s_T$,
then it holds that $F(\tau)<\epsilon_1$ for $\tau>s$.
In fact,
we assume that there exists
$\tau_1>s$ such that $F(\tau)<\epsilon_1$ for $s<\tau<\tau_1$ and $F(\tau_1)=\epsilon_1$.
By definition of $\epsilon_1$ and \eqref{eq2C},
we find that $\tau_1>s+1$.
Therefore
we get from definition of $\tau_1$ and \eqref{eq2D} that
\[
 \|u(\tau)\|_{L^\infty(-R_0,R_0)}+\|v(\tau)\|_{L^\infty(-R_0,R_0)}
<
 c_3\epsilon_1
<
 \frac{\epsilon_0}{2}
\]
for $\tau\in(s+1,\tau_1)$.
As a consequence,
from definition of $R_0$ and $\epsilon_0$,
it follows that $F_s(s)<0$ for $s\in(s+1,\tau_1)$.
However
this contradicts definition of $\tau_1$,
which completes the proof.
\end{proof}
%%%%%%%%%%%%%%%%%%%%%%%%%%%%%%%%%%%%%%%%%%%%%%%%%%%%%%%%%%%%%%%%

%%%%%%%%%%%%%%%%%%%%%%%%%%%%%%%  Lem  %%%%%%%%%%%%%%%%%%%%%%%%%%%%%%%
\begin{lem}\label{2BLem}
Let $(a,b)$ and $(u_\xi,v_\xi)$ be as in Lemma {\rm\ref{2ALem}}
and
$\{\xi_i\}_{i\in\N}\subset\R$, $\{s_i\}_{i\in\infty}$ $(s_i\to\infty)$ be sequences
and put
\[
 u_i(y,s) = u_{\xi_i}(y,s_i+s),
\hfive
 v_i(y,s) = u_{\xi_i}(y,s_i+s).
\]
Then
if $(u_i,v_i)\to(U,V)$ in $L_{\mathrm{loc}}^\infty(\R\times\R)$ as $i\to\infty$
and $(U(s),V(s))\to(0,0)$ in $L_{\mathrm{loc}}^\infty(\R)$,
then $\xi_i\in\R$ is not a blow-up point of $(a,b)$ for large $i\in\N$.
\end{lem}
%%%%%%%%%%%%%%%%%%%%%%%%%%%%%%%  Proof  %%%%%%%%%%%%%%%%%%%%%%%%%%%%%%%
\begin{proof}
Let $R>0$ and $\epsilon_0>0$ be given in Lemma \ref{2ALem}.
Since $(U(s),V(s))\to(0,0)$,
there exists $s_*>0$ such that
$\|U(s_*)\|_{L^\infty(-2R,2R)}+\|V(s_*)\|_{L^\infty(-2R,2R)}<\epsilon_0/2$.
Furthermore
since $(u_i,v_i)\to(U,V)$ as $i\to\infty$,
we see that
\[
\begin{array}{l}
\dis
 \|u_{\xi_i}(s_i+s_*)\|_{L^\infty(-R,R)}+\|v_{\xi_i}(s_i+s_*)\|_{L^\infty(-R,R)}
\\[2mm] \dis \hfive\hfive
=
 \|u_i(s_*)\|_{L^\infty(-R,R)}+\|v_i(s_*)\|_{L^\infty(-R,R)}<\epsilon_0
\end{array}
\]
for large $i\in\N$.
Therefore from Lemma \ref{2ALem},
$\xi_i$ is not a blow-up point of $(a,b)$,
which completes the proof.
\end{proof}
%%%%%%%%%%%%%%%%%%%%%%%%%%%%%%%%%%%%%%%%%%%%%%%%%%%%%%%%%%%%%%%%

%%%%%%%%%%%%%%%%%%%%%%%%%%%%%%%  Lem  %%%%%%%%%%%%%%%%%%%%%%%%%%%%%%%
\begin{lem}\label{2CLem}
Let $(a_i,b_i)$ be a solution of \eqref{eq1A} and satisfies
$\sup_{x\in\R}(|a_i(x,t)|+|b_i(x,t)|)<c/(1-t)$ for $t\in(0,1)$.
If $(a_i,b_i)\to(A,B)$ and $(A,B)$ does not blow up on $x=x_0$ at $t=1$,
then $x_0$ is not a blow-up point of $(a_i,b_i)$ at $t=1$
for large $i\in\N$.
\end{lem}
%%%%%%%%%%%%%%%%%%%%%%%%%%%%%%%  Proof  %%%%%%%%%%%%%%%%%%%%%%%%%%%%%%%
\begin{proof}
Set $1-t=e^{-s}$,
$u_i(y,s)=(1-t)a_i(x_0+e^{-s/2}y,t)$ and $v_i(y,s)=(1-t)b_i(x_0+e^{-s/2}y,t)$.
From the assumption,
we see that $(u_i,v_i)$ is uniformly bounded on $\R\times(0,\infty)$.
% Furthermore
% it holds that 
Since $(a_i,b_i)\to(A,B)$ and $u_i(y,0)=a_i(x_0+y,0)$, $v_i(y,0)=b_i(x_0+y,0)$,
we see that
$(u_i,v_i)\to(U,V)$ and $U(y,s)=(1-t)A(x_0+e^{-s/2}y,t)$, $V(y,s)=(1-t)B(x_0+e^{-s/2}y,t)$
for $s>0$.
% The boundedness of $(u_i,v_i)$ implies that
% $(U,V)$ is uniformly bounded on $\R\times(0,\infty)$
Since $\sup_{x\in\R}(|A(x,t)|+|B(x,t)|)<c/(1-t)$,
$(A,B)$ does not blow up for $t\in(0,1)$.
If $(A,B)$ does not blow up on $x=x_0$ at $t=1$,
it holds that $(U,V)\to(0,0)$ as $s\to\infty$.
Therefore by the same way as in the proof of Lemma \ref{2BLem},
we conclude that
$x_0$ is not a blow-up point of $(a_i,b_i)$ at $t=1$ for large $i\in\N$.
The proof is completed.
\end{proof}
%%%%%%%%%%%%%%%%%%%%%%%%%%%%%%%%%%%%%%%%%%%%%%%%%%%%%%%%%%%%%%%%

%%%%%%%%%%%%%%%%%%%%%%%%%%%%%%%%%%%%%%%%%%%%%%%%%%%%%%%%%%%%%%%%
\section{Local conditions for boundedness of solutions}
In this section,
we provide the proof of Theorem \ref{thm2}.
Let $x_0\in\R$ be a blow-up point, $T>0$ be a blow-up time and
${\cal O}$ be the neighborhood of $(x_0,T)$ stated in Theorem \ref{thm2}.
Since
the proof for the case $b(x,t)<0$ for $(x,t)\in{\cal O}$ is the same as
for the case $b(x,t)>0$ for $(x,t)\in{\cal O}$,
we here only consider the later case.
For such a case,
by shifting the initial time,
we can assume
\[
 b(x,t)>0 \hfive \text{for } x\in(L_1,L_2),\ t\in(0,T)
\]
for some $L_1<x_0<L_2$.
Furthermore throughout this section, we assume \eqref{eq1C}.
%%%%%%%%%%%%%%%%%%%%%%%%%%%%%%%%%%%%%%%%%%%%%%%%%%%%%%%%%%%%%%%%

%%%%%%%%%%%%%%%%%%%%%%%%%%%%%%%  Lem  %%%%%%%%%%%%%%%%%%%%%%%%%%%%%%%
\begin{lem}\label{3ALem}
Either one of the intervals $(L_1,x_0)$ and $(x_0,L_2)$ is included in the blow-up set.
\end{lem}
%%%%%%%%%%%%%%%%%%%%%%%%%%%%%%%  Proof  %%%%%%%%%%%%%%%%%%%%%%%%%%%%%%%
\begin{proof}
Assume that their exist $l_1\in(L_1,x_0)$ and $l_2\in(x_0,L_2)$ such that
$x=l_1$ and $x=l_2$ are not blow-up points.
From this assumption,
$a(x,t)$ and $b(x,t)$ are uniformly bounded on $(l_1-\epsilon,l_1+\epsilon)$ and
$(l_2-\epsilon,l_2+\epsilon)$ for some $\epsilon>0$.
Therefore
since $b(x,t)>0$ in $(L_1,L_2)$,
by a comparison argument,
we easily see that
\begin{equation}\label{eq3B}
 b(l_1,t)>\delta \hthree \text{for } t\in(0,T),
\hfive
 b(l_2,t)>\delta \hthree \text{for } t\in(0,T),
\hfive
 b_0(x)>\delta \hthree \text{for } x\in(l_1,l_2)
\end{equation}
with some $\delta>0$.
Set
$\gamma=a/b$.
Then $\gamma$ satisfies
\[
 \gamma_t = \gamma_{xx}+2\nu\gamma_x-\left( \frac{a^2+b^2}{b} \right),
\]
where $\nu=b_x/b$.
Since $x=l_1$ and $x=l_2$ are not blow-up points,
it is clear that $M=\sup_{0<t<T}(|a(l_1,t)|+|a(l_2,t)|+|a_0(x)|)<\infty$.
Combining this fact and \eqref{eq3B},
we get
\[
 \gamma(l_1,t) < M/\delta  \hthree \text{for } t\in(0,T),
\hfive
 \gamma(l_2,t) < M/\delta  \hthree \text{for } t\in(0,T),
\hfive
 \gamma_0(x) < M/\delta  \hthree \text{for } x\in(l_1,l_2).
\]
Therefore
we obtain from a maximum principle that
\begin{equation}\label{eq3C}
 \gamma(x,t) > M/\delta
\hfive \text{for } x\in(l_1,l_2),\ t\in(0,T).
\end{equation}
% The assumption \eqref{eq1C} implies that
% there exists $c_0>0$ such that
% \begin{equation}\label{eq3D}
%  (T-t)\sup_{x\in\R}(|a(x,t)|+|b(x,t)|) < c_0
% \hfive \text{for } t\in(0,T).
% \end{equation}
Let $\lambda_i=1/i$ and set
$a_i(x,\tau)=\lambda_ia(x_0+\sqrt{\lambda_i}x,T-1/i+\lambda_i\tau)$ and
$b_i(x,\tau)=\lambda_ib(x_0+\sqrt{\lambda_i}x,T-1/i+\lambda_i\tau)$.
Then
we easily see that \eqref{eq1C} is equivalent to
\[
 \sup_{x\in\R}(|a_i(x,\tau)|+|b_i(x,\tau)|) < \frac{c_0}{1-\tau}.
\]
Therefore
by taking a subsequence,
we get $(a_i,b_i)\to(A,B)$ and
\[
 \sup_{x\in\R}(|A(x,\tau)|+|B(x,\tau)|) < \frac{c_0}{1-\tau}.
\hfive \text{for } \tau\in(-1,1).
\]
Furthermore
we get from \eqref{eq3C} that
\[
 A(x,\tau)/B(x,\tau)<M/\delta
\hfive \text{for } x\in\R,\ \tau\in(-1,1).
\]
Since $(A,B)$ is a solution of \eqref{eq1A},
Theorem 1.1 \cite{Guo} stated in Introduction implies
that $(A,B)$ exists globally in time.
Therefore
from Lemma \ref{2CLem},
the origin is not a blow-up point of $(a_i,b_i)$ for large $i\in\N$,
which implies that $x_0$ is not a blow-up point of $(a,b)$.
This contradicts the assumption,
which completes he proof. 
\end{proof}
%%%%%%%%%%%%%%%%%%%%%%%%%%%%%%%%%%%%%%%%%%%%%%%%%%%%%%%%%%%%%%%%

%%%%%%%%%%%%%%%%%%%%%%%%%%%%%%%%%%%%%%%%%%%%%%%%%%%%%%%%%%%%%%%%
From Lemma \ref{3ALem},
we can assume that
the interval $(-L,L)$ is included in the blow-up set and $b$ satisfies
\begin{equation}\label{eq3F}
 b(x,t)>0
\hfive \text{for } x\in(-L,L),\ t\in(0,T).
\end{equation}
We now introduce self-similar variables and define a new unknown function $(u,v)$
as in Section \ref{Sub21}.
Let $\xi\in\R$ and set
\[
\begin{array}{c}
 T-t=e^{-s},
\hfive
 s_T=-\log(T-t),
\\[2mm]
 u_\xi(y,s)=(T-t)a(\xi+e^{-s/2}y,t),
\hfive
 v_\xi(y,s)=(T-t)b(\xi+e^{-s/2}y,t).
\end{array}
\]
Then $(u,v)=(u_\xi,v_\xi)$ satisfies \eqref{eq2B}.
%%%%%%%%%%%%%%%%%%%%%%%%%%%%%%%%%%%%%%%%%%%%%%%%%%%%%%%%%%%%%%%%

%%%%%%%%%%%%%%%%%%%%%%%%%%%%%%%  Lem  %%%%%%%%%%%%%%%%%%%%%%%%%%%%%%%
\begin{lem}\label{3BCLem}
Let $\{\xi_i\}_{i\in\N}\subset(-L/2,L/2)$ and $\{s_i\}_{i\in\N}$ $(s_i\to\infty)$
be sequences.
Put
\[
\begin{array}{c}
 a_i(x,\tau) = \lambda_ia(\xi_i+\sqrt{\lambda_i}x,t_i+\lambda_i\tau),
\hfive
 b_i(x,\tau) = \lambda_ib(\xi_i+\sqrt{\lambda_i}x,t_i+\lambda_i\tau).
% \\[2mm] \dis
%  u_i(y,s) = u_{\xi_i}(y,s_i+s),
% \hfive
%  v_i(y,s) = v_{\xi_i}(y,s_i+s).
\end{array}
\]
If $(a_i,b_i)\to(A,B)$ as $i\to\infty$ and $(A,B)$ blows up on the origin at $\tau=1$,
then the origin is not an isolated blow-up point of $(A,B)$.
\end{lem}
%%%%%%%%%%%%%%%%%%%%%%%%%%%%%%%  Proof  %%%%%%%%%%%%%%%%%%%%%%%%%%%%%%%
\begin{proof}
We prove by contradiction.
Assume that the origin is an isolated blow-up point of $(A,B)$.
Then
there exist $\theta_1,\theta_2$ ($0<\theta_1<\theta_2<1$) such that
\[
 \sup_{0<\tau<1}\sup_{\theta_1<x<\theta_2}(|A(x,\tau)|+|B(x,\tau)|) < \infty.
\]
Therefore from Lemma \ref{2CLem},
there exists $c>0$ such that
\begin{equation}\label{eq3I}
 \sup_{0<\tau<1}\sup_{\theta_1<x<\theta_2}(|a_i(x,\tau)|+|b_i(x,\tau)|) < c
\hfive \text{for } i\gg1.
\end{equation}
Let $\theta=(\theta_1+\theta_2)/2$ and
\[
 \tilde{u}_i(y,s)=(1-\tau)a_i(\theta+e^{-s/2}y,\tau),
\hfive
 \tilde{v}_i(y,s)=(1-\tau)b_i(\theta+e^{-s/2}y,\tau),
\hfive
 1-\tau=e^{-s}.
\]
Then
we see that
\begin{align*}
 \tilde{u}_i(y,s)
&=
 e^{-(s+s_i)}a(\tilde{\xi}_i+e^{-(s+s_i)/2}y,T-e^{-(s+s_i)}) =
 u_{\tilde{\xi}_i}(y,s_i+s),
\\
 \tilde{v}_i(y,s)
&=
 v_{\tilde{\xi}_i}(y,s_i+s),
\end{align*}
where $\tilde{\xi}_i=\xi_i+\sqrt{\lambda_i}\theta$.
Put $\Delta=(\theta_2-\theta_1)/2$.
Then we get from \eqref{eq3I} that
\begin{align*}
 \sup_{|y|<e^{s/2}\Delta}(|u_{\tilde{\xi}_i}(y,s_i+s)|+|v_{\tilde{\xi}_i}(y,s_i+s)|)
&=
 \sup_{|y|<e^{s/2}\Delta}(|\tilde{u}_i(y,s)|+|\tilde{v}_i(y,s)|)
\\
&=
 \sup_{\theta_1<x<\theta_2}e^{-s}(|a_i(x,\tau)|+|b_i(x,\tau)|)
\\
&<
 ce^{-s}
\hfive \text{for } s>0,\ i\gg1.
\end{align*}
This implies
\[
 \sup_{|y|<e^{s/2}\Delta}(|u_{\tilde{\xi}_i}(y,s)|+|v_{\tilde{\xi}_i}(y,s)|)
<
 ce^{-(s-s_i)}
\hfive\text{for } s>s_i,\ i\gg1.
\]
Therefore from Lemma \ref{2ALem},
$\tilde{\xi}_i$ is not a blow-up point of $(a,b)$,
which contradicts that $\tilde{\xi}_i$ is a blow-up point of $(a,b)$.
\end{proof}
%%%%%%%%%%%%%%%%%%%%%%%%%%%%%%%%%%%%%%%%%%%%%%%%%%%%%%%%%%%%%%%%

%%%%%%%%%%%%%%%%%%%%%%%%%%%%%%%  Lem  %%%%%%%%%%%%%%%%%%%%%%%%%%%%%%%
\begin{lem}\label{3CLem}
For any $R>0$,
there exists $\epsilon_1>0$ such that
if $\dis\inf_{-R<y<R}v_\xi(y,s)<\epsilon_1$,
then it holds that
\[
 \sup_{-R<y<R}|u_\xi(y,s)-1|<1/8
\hfive\mathrm{for}\ s>s_T,\ \xi\in(-L/2,L/2).
\]
\end{lem}
%%%%%%%%%%%%%%%%%%%%%%%%%%%%%%%  Proof  %%%%%%%%%%%%%%%%%%%%%%%%%%%%%%%
\begin{proof}
We prove by contradiction. 
Assume that
there exist $R>0$, $\{s_i\}_{i\in\N}$ ($s_i\to\infty$) and $\{\xi_i\}_{i\in\N}\subset(-L/2,L/2)$
such that
\begin{equation}\label{eq3J}
 \inf_{-R<y<R}v_{\xi_i}(y,s_i)<1/i,
\hfive\hfive
 \sup_{-R<y<R}|u_{\xi_i}(y,s_i)-1|>1/8.
\end{equation}
Put $\lambda_i = e^{-s_i}$, $t_i=T-\lambda_i$ and
\[
% \begin{array}{c}
%  u_i(y,s) = u_{\xi_i}(y,s_i+s),
% \hfive
%  v_i(y,s) = v_{\xi_i}(y,s_i+s),
% \\[2mm] \dis
 a_i(x,\tau) = \lambda_ia(\xi_i+\sqrt{\lambda_i}x,t_i+\lambda_i\tau),
\hfive
 b_i(x,\tau) = \lambda_ib(\xi_i+\sqrt{\lambda_i}x,t_i+\lambda_i\tau).
% \end{array}
\]
Then
we easily see from \eqref{eq1C} that
\begin{equation}\label{eq3II}
 |a_i(x,\tau)|+|b_i(x,\tau)|<c(1-\tau)^{-1}
\end{equation}
for some $c>0$.
Therefore
by taking a subsequence,
we get $(a_i,b_i)\to(A,B)$.
Since $b_i(x,0)=v_{\xi_i}(x,s_i)$,
by a strong maximum principle,
$B$ must be  zero on $\R\times(0,1)$.
If $A\equiv0$,
Lemma \ref{2CLem} implies that $(a_i,b_i)$ does not blow up on the origin.
Therefore
it is sufficient to consider the case $A\not\equiv0$ on $\R\times(0,1)$.
We note from \eqref{eq3II} that $A$ exists at least until $\tau=1$.
Since the origin is a blow-up point of $(a_i,b_i)$ at $\tau=1$,
$A$ must blow up at the origin at $\tau=1$ from Lemma \ref{2CLem}.
Since $B\equiv0$,
$A$ satisfies $A_\tau=A_{xx}+A^2$.
From Theorem \cite{Herrero} p.209,
there are two possibilities:
(I) $A\equiv1$ or (II) the origin is the isolated blow-up point.
Since (II) is excluded from Lemma \ref{3BCLem},
(I) occurs.
Therefore
this contradicts \eqref{eq3J},
which completes the proof,
\end{proof}
%%%%%%%%%%%%%%%%%%%%%%%%%%%%%%%%%%%%%%%%%%%%%%%%%%%%%%%%%%%%%%%%

%%%%%%%%%%%%%%%%%%%%%%%%%%%%%%%  Lem  %%%%%%%%%%%%%%%%%%%%%%%%%%%%%%%
\begin{lem}\label{3DLem}
Let $v_\pm=v_{\xi}$ with $\xi=\pm L/4$.
Then it holds that
\[
 \liminf_{s\to\infty}v_{\pm}(0,s)>0.
\]
\end{lem}
%%%%%%%%%%%%%%%%%%%%%%%%%%%%%%%  Proof  %%%%%%%%%%%%%%%%%%%%%%%%%%%%%%%
\begin{proof}
Let ${\cal A}=\pa_y^2-\frac{y}{2}\pa_y$.
Since the first eigenvalue of ${\cal A}$ in $H_\rho^1(\R)$ is zero,
we can choose $R_0>0$ such that
the first eigenvalue of ${\cal A}|_{\text{Dirichlet}}$
in $H_{\rho}^1(-R_0,R_0)=\{f\in H_\rho^1(\R);f(y)=0$ for $|y|>R_0\}$
is less than $1/8$.
Put $v_\pm=v_{\xi}$ with $\xi=\pm L/4$.
From \eqref{eq3F},
we see that $v_\pm$ is positive on $(-R_0,R_0)$ for large $s>s_T$.
Let $\phi(y)>0$ be the first eigenfunction of ${\cal A}|_{\text{Dirichlet}}$.
Then from Lemma \ref{3CLem},
if we choose $\epsilon>0$ sufficiently small,
$\psi=\epsilon\phi$ becomes a subsolution of $v_\pm$ in $(-R_0,R_0)$,
which completes the proof.
\end{proof}
%%%%%%%%%%%%%%%%%%%%%%%%%%%%%%%%%%%%%%%%%%%%%%%%%%%%%%%%%%%%%%%%

%%%%%%%%%%%%%%%%%%%%%%%%%%%%%%%  proof  %%%%%%%%%%%%%%%%%%%%%%%%%%%%%%%
\begin{proof}[Proof of Theorem {\rm\ref{thm2}}]
Combining Lemma \ref{3DLem} and \eqref{eq1C},
we obtain $a(\pm L/4,t)/b(\pm L/4,t)<c'$ for some $c'>0$.
Therefore
by the same argument as in the proof of Lemma \ref{3ALem},
we see that the origin is not a blow-up point,
which contradicts the assumption.
The proof of Theorem \ref{thm2} is completed.
\end{proof}
%%%%%%%%%%%%%%%%%%%%%%%%%%%%%%%%%%%%%%%%%%%%%%%%%%%%%%%%%%%%%%%%

%%%%%%%%%%%%%%%%%%%%%%%%%%%%%%%%%%%%%%%%%%%%%%%%%%%%%%%%%%%%%%%%
\section{Location of blow-up points}
This section is devoted to the proof of Theorem \ref{thm3}.
From Theorem \ref{thm2},
if a solution of \eqref{eq1A} blows up in a finite time,
$b$ must be sign changing near the blow-up point.
Here
we discuss a relation between blow-up points and zeros of $b$.
Since $b$ satisfies $b_t=b_{xx}+2ab$,
the number of zeros of $b(t)$ is nonincreasing in $t$ (see e.g. \cite{Matano}). 
Therefore from assumption of Theorem \ref{thm3},
the number of zeros of $b(t)$ is one or zero for $t\in(0,T)$.
However
since $(a,b)$ blows up at $t=T$, 
$b(t)$ has one zero for $t\in(0,T)$ from Theorem \ref{thm2}.
Throughout this section,
we assume that $b(t)$ has one zero for $t\in(0,T)$
and denote a zero of $b(t)$ by $\gamma(t)$.
Furthermore
we assume
\[
 b(x,t) =
 \begin{cases}
 \text{negative} & \text{if } x<\gamma(t)
 \\
 \text{positive} & \text{if } x>\gamma(t)
 \end{cases}
\]
%%%%%%%%%%%%%%%%%%%%%%%%%%%%%%%%%%%%%%%%%%%%%%%%%%%%%%%%%%%%%%%%

%%%%%%%%%%%%%%%%%%%%%%%%%%%%%%%  Pro  %%%%%%%%%%%%%%%%%%%%%%%%%%%%%%%
\begin{pro}\label{5APro}
Let $x_0\in\R$ be an isolated blow-up point.
Then the blow-up set on $\R$ consists of $x_0$.
\end{pro}
%%%%%%%%%%%%%%%%%%%%%%%%%%%%%%%  Proof  %%%%%%%%%%%%%%%%%%%%%%%%%%%%%%%
\begin{proof}
To derive contradiction,
we assume that $x_1>x_0$ is another blow-up point.
Since $x_0$ and $x_1$ are blow-up points,
we see from Theorem \ref{thm2} that
\begin{equation}\label{5Aeq}
 \liminf_{t\to T} \gamma(t) \leq x_0,
\hfive
 \limsup_{t\to T} \gamma(t) \geq x_1.
\end{equation}
Let $x_2=(x_0+x_1)/2$, $\delta=(x_1-x_0)/2$ and set
\[
 u(y,s) = e^{-s}a(x_2+e^{-s/2}y,T-e^{-s}),
\hfive
 v(y,s) = e^{-s}b(x_2+e^{-s/2}y,T-e^{-s}).
\]
% Then $(u,v)$ satisfies \eqref{eq2B}.
Since $(a,b)$ is uniformly bounded on $(x_2-\delta,x_2+\delta)$,
$(u,v)$ satisfies
\[
 \sup_{|y|<\delta e^{s/2}}(|u(y,s)|+|v(y,s)|)
<
 c_1e^{-s}
\]
for some $c_1>0$.
Therefore
we get from \eqref{5FFGeq} and \eqref{eq1C}
\begin{align*}
 \int_{-\infty}^\infty|uv|^2\rho dy
&<
 c_1^2e^{-2s}\int_{|y|<\delta e^{s/2}}|v|^2\rho dy
+ 
 \delta^{-2}e^{-s}\|u\|_{L^\infty(\R)}^2\int_{|y|>\delta e^{s/2}}|y|^2|v|^2\rho dy
\\
&<
 \left( c_1^2+\delta^{-2} \right)e^{-s}\|v\|_{H^1_\rho(\R)}^2.
\end{align*}
Let ${\cal A}=\pa_y^2-\frac{y}{2}\pa_y$.
Then we see that
\[
 \|v_s-({\cal A}-1)v\|_\rho
<
 2\|uv\|_\rho
<
 2\sqrt{c_1^2+\delta^{-2}}e^{-s/2}\|v\|_{H^1_\rho(\R)}.
\]
Therefore
from Lemma A.16 \cite{Ackermann} (see also \cite{Cohen,Ogawa}),
we obtain $\|v(s)\|_\rho \geq ce^{-\mu s}$ for some $\mu>0$.
As a consequence,
there exists $k\in\N$ such that
\[
 v(s) = \alpha_k(1+o(1))e^{-\lambda_ks}\phi_k
\hfive \text{in } L_\rho^2(\R).
\]
However
this contradicts \eqref{5Aeq},
which completes the proof.
\end{proof}
%%%%%%%%%%%%%%%%%%%%%%%%%%%%%%%%%%%%%%%%%%%%%%%%%%%%%%%%%%%%%%%%

%%%%%%%%%%%%%%%%%%%%%%%%%%%%%%%%%%%%%%%%%%%%%%%%%%%%%%%%%%%%%%%%
Let $x_0\in\R$ be a blow-up point of $(a,b)$.
If $x_0$ is an isolated blow-up point,
Proposition \ref{5APro} implies that
no other blow-up points exist on $\R$.
Then
we see that $\gamma(t)$ is continuous on $(0,T]$.
In fact,
if $\gamma$ is not continuous at $t=T$,
% we get from Theorem \ref{thm2}
it satisfies
\[
 \liminf_{t\to T}\gamma(t) < \liminf_{t\to T}\gamma(t).
\]
However by the same argument as in the proof of Lemma \ref{5APro},
we derive contradiction.
Therefore
if $x_0$ is an isolated blow-up point of $(a,b)$,
the proof is completed.
We here consider the case where there are no isolated blow-up points.
Let $x_1>x_0$ be another blow-up point.
Then the interval $(x_0,x_1)$ is included in the blow-up set.
By shifting the origin,
we can assume that
\begin{equation}\label{5Beq}
 \text{the interval $(-L,L)$ is included in the blow-up set}.
\end{equation}
We put $e^{-s}=T-t$ and
\[
 u(y,s) = (T-t)a(e^{-s/2}y,t),
\hfive
 v(y,s) = (T-t)b(e^{-s/2}y,t).
\]
We denote a zero of $v(s)$ by $\Gamma(s)$,
which satisfies $\Gamma(s)=e^{s/2}\gamma(t)$.
%%%%%%%%%%%%%%%%%%%%%%%%%%%%%%%%%%%%%%%%%%%%%%%%%%%%%%%%%%%%%%%%

%%%%%%%%%%%%%%%%%%%%%%%%%%%%%%%  Lem  %%%%%%%%%%%%%%%%%%%%%%%%%%%%%%%
\begin{lem}\label{5BLem}
For any $\epsilon_0>0$ there exists $K>0$ such that
if $|v(y,s)|>\epsilon_0$ for some $|y|<s$ and $s\gg1$,
then it holds that $|y-\Gamma(s)|<K$. 
\end{lem}
%%%%%%%%%%%%%%%%%%%%%%%%%%%%%%%  Proof  %%%%%%%%%%%%%%%%%%%%%%%%%%%%%%%
\begin{proof}
We prove by contradiction.
Assume that there exist $\epsilon_0>0$, $\{y_i\}_{i\in\N}$ and $\{s_i\}_{i\in\N}$
satisfying $|y_i|<s_i$ and $s_i\to\infty$
such that
\begin{equation}\label{5Ceq}
 |v(y_i,s_i)|>\epsilon_0,
\hfive\hfive
 |y_i-\Gamma(s_i)|>i.
\end{equation}
We put $\lambda_i=e^{-s_i}$, $t_i=T-e^{-s_i}$ and
\[
 a_i(x,\tau) = \lambda_ia(\sqrt{\lambda_i}y_i+\sqrt{\lambda_i}x,t_i+\lambda_i\tau),
\hfive
 b_i(x,\tau) = \lambda_ib(\sqrt{\lambda_i}y_i+\sqrt{\lambda_i}x,t_i+\lambda_i\tau).
\]
Then
\eqref{eq1C} implies
\begin{equation}\label{5Deq}
 \sup_{x\in\R}(|a_i(x,\tau)|+|b_i(x,\tau)|) < \frac{c_1}{1-\tau}
\hfive\text{for } \tau\in(0,1)
\end{equation}
with some $c_1>0$.
Furthermore
we easily see that $a_i(x,0) = u(y_i+x,s_i)$ and $b_i(x,0) = v(y_i+x,s_i)$.
Therefore
it follows from \eqref{5Ceq} that
\begin{equation}\label{5Eeq}
 |b_i(x,0)|>0
\hfive\text{for } |x|<i.
\end{equation}
% Furthermore
% by the same argument as in the proof of Lemma \ref{3CLem},
By taking a subsequence,
we get
\[
 (a_i,b_i) \to (A,B).
% \hfive \R\times[0,1)\text{上広義一様収束}.
\]
Then
we get from \eqref{5Deq} and \eqref{5Eeq},
\[
 |B(x,0)|>0 \hthree\text{for } x\in\R,
\hfive
 \sup_{x\in\R}(|A(x,\tau)|+|B(x,\tau)|) < \frac{c_1}{1-\tau} 
\hthree\text{for } \tau\in(0,1).
\]
From Theorem \ref{thm2},
we find that $(A,B)$ does not blow up on the origin at $\tau=1$.
As a consequence,
from Lemma \ref{2CLem},
the origin is not a blow-up point of $(a_i,b_i)$ at $\tau=1$ for large $i\in\N$,
which implies that $\sqrt{\lambda_i}y_i$ is not a blow-up point of $(a,b)$ for large $i\in\N$.
However since $\sqrt{\lambda_i}y_i\to0$ as $i\to\infty$,
this contradicts \eqref{5Beq}.
\end{proof}
%%%%%%%%%%%%%%%%%%%%%%%%%%%%%%%%%%%%%%%%%%%%%%%%%%%%%%%%%%%%%%%%

%%%%%%%%%%%%%%%%%%%%%%%%%%%%%%%  Lem  %%%%%%%%%%%%%%%%%%%%%%%%%%%%%%%
\begin{lem}\label{5CLem}
For any $\delta>0$ and $r>0$
there exists $m_0>0$ such that
if $\dis\|v(s)\|_{L^\infty(-1,1)}<m_0$ for some $s\gg1$,
then it holds that
\[
 \sup_{-r<y<r}\left( |u(y,s)-1|+|u_y(y,s)| \right) < \delta.
\]
\end{lem}
%%%%%%%%%%%%%%%%%%%%%%%%%%%%%%%  Proof  %%%%%%%%%%%%%%%%%%%%%%%%%%%%%%%
\begin{proof}
Since the proof of this lemma is the same as that of Lemma \ref{3CLem},
we omit the detail.
\end{proof}
%%%%%%%%%%%%%%%%%%%%%%%%%%%%%%%%%%%%%%%%%%%%%%%%%%%%%%%%%%%%%%%%

%%%%%%%%%%%%%%%%%%%%%%%%%%%%%%%  Lem  %%%%%%%%%%%%%%%%%%%%%%%%%%%%%%%
\begin{lem}\label{5DLem}
$\dis\liminf_{s\to\infty}\|v(s)\|_{L^\infty(-1,1)}=0$.
\end{lem}
%%%%%%%%%%%%%%%%%%%%%%%%%%%%%%%  Proof  %%%%%%%%%%%%%%%%%%%%%%%%%%%%%%%
\begin{proof}
Since the interval $(-L,L)$ is included in the blow-up set,
we get from Theorem \ref{thm2} that
\[
 \liminf_{t\to T}\gamma(t)\leq-L,
\hfive
 \limsup_{t\to T}\gamma(t)\geq L.
\]
Therefore
since $\Gamma(s)=e^{s}\gamma(t)$,
Lemma \ref{5BLem} proves this lemma.
\end{proof}
%%%%%%%%%%%%%%%%%%%%%%%%%%%%%%%%%%%%%%%%%%%%%%%%%%%%%%%%%%%%%%%%

%%%%%%%%%%%%%%%%%%%%%%%%%%%%%%%  Pro  %%%%%%%%%%%%%%%%%%%%%%%%%%%%%%%
\begin{pro}\label{5BPro}
$\dis\lim_{s\to\infty}\|v(s)\|_{L^\infty(-1,1)}=0$.
\end{pro}
%%%%%%%%%%%%%%%%%%%%%%%%%%%%%%%%%%%%%%%%%%%%%%%%%%%%%%%%%%%%%%%%
The proof of this Proposition is given in Section \ref{5Sub},
which is a crucial step in this paper.
%%%%%%%%%%%%%%%%%%%%%%%%%%%%%%%%%%%%%%%%%%%%%%%%%%%%%%%%%%%%%%%%

%%%%%%%%%%%%%%%%%%%%%%%%%%%%%%%%%%%%%%%%%%%%%%%%%%%%%%%%%%%%%%%%
\subsection{Proof of Proposition \ref{5BPro}}\label{5Sub}
%%%%%%%%%%%%%%%%%%%%%%%%%%%%%%%%%%%%%%%%%%%%%%%%%%%%%%%%%%%%%%%%
This proof is based on the argument in \cite{Filippas}.
We carefully investigate the behavior of solutions
through a dynamical system approach in $L_\rho^2(\R)$.
Since $v(s)$ has exact one zero for $s>s_T$,
we focus on the behavior of the corresponding eigenmode of $v(s)$.
%%%%%%%%%%%%%%%%%%%%%%%%%%%%%%%%%%%%%%%%%%%%%%%%%%%%%%%%%%%%%%%%

%%%%%%%%%%%%%%%%%%%%%%%%%%%%%%%%%%%%%%%%%%%%%%%%%%%%%%%%%%%%%%%%
\subsubsection{Choice of $\bar{\eta}$ $\bar{\zeta}$ $\bar{\epsilon}$ $\bar{\delta}$, $\bar{R}$}
%%%%%%%%%%%%%%%%%%%%%%%%%%%%%%%%%%%%%%%%%%%%%%%%%%%%%%%%%%%%%%%%
Let ${\cal A} = \pa_{yy}-\frac{y}{2}\pa_y$.
It is known that
$H_\rho^1(\R)$ is spanned by eigenfunctions $\{\phi_i\}_{i\in\N}$ of ${\cal A}$.
A function $v$ in $H_\rho^1(\R)$ is decomposed to
\[
 v = \alpha\phi_0+\beta\phi_1+\gamma \phi_2+w.
\]
Since $\phi_2(y)=c_1(y^2-1)$ for some $c_1>0$,
it follows that $\phi_2(0)=-c_1$ and $\phi_2(2)=3c_1$.
Here we recall the inequality:
$\|w\|_{L^\infty(-2,2)} < c\|w\|_{H_\rho^1(\R)}$.
Therefore
there exists $\epsilon_1>0$ such that
if  $v\in H_\rho^1(\R)$ satisfies $\alpha^2+\beta^2+\|w\|_{H_\rho^1(\R)}^2 < \epsilon_1\gamma^2$,
then $v$ has at least two zeros in  $(-2,2)$.
Here we fix $\bar{\eta}>0$, $\bar{\zeta}>0$ and $\bar{\epsilon}\in(0,1/4)$ such that
\begin{equation}\label{5Feq}
 2\left( \frac{1+\bar{\zeta}}{\bar{\eta}}+\bar{\zeta} \right) < \epsilon_1,
\hfive
 \bar{\epsilon}
 \left( \frac{1}{\bar{\eta}}\left( \frac{1}{\bar{\zeta}}+1 \right)+\frac{1}{\bar{\zeta}} \right)
< \frac{1}{8},
\hfive
 \left( \frac{1}{4}-2\bar{\epsilon} \right)\bar{\eta}-\left( 2+\bar{\eta}^2 \right)\bar{\epsilon}
>0,
\hfive
 \bar{\epsilon}\bar{\eta}<\frac{1}{8}.
\end{equation}
Furthermore
we put
\[
 \bar{M} = \sup_{y\in\R,s>s_T}(|u(y,s)-1|+|u_y(y,s)|).
\]
By using \eqref{5FFGeq},
we can fix $\bar{\delta}>0$ and $\bar{R}>0$ such that
if $|P(y)|<\bar{\delta}$ for $|y|<\bar{R}$ and $\|P\|_{L^\infty(\R)}<\bar{M}$,
then it holds that
\[
 \int_{-\infty}^\infty
 P(y)^2
 \left( \sum_{k=0}^2\left(|\phi_k|^2+|\phi_k'|^2\right) \right)
 \rho dy
<
 \left( \frac{\bar{\epsilon}}{24} \right)^2,
\hfive\hfive
 \int_{-\infty}^\infty|P(y)|v^2\rho dy
<
 \frac{\bar{\epsilon}}{8}\|v\|_{H_\rho^1(\R)}^2.
\]
%%%%%%%%%%%%%%%%%%%%%%%%%%%%%%%%%%%%%%%%%%%%%%%%%%%%%%%%%%%%%%%%

%%%%%%%%%%%%%%%%%%%%%%%%%%%%%%%%%%%%%%%%%%%%%%%%%%%%%%%%%%%%%%%%
\subsubsection{Assumptions and setting}
To prove Proposition \ref{5BPro},
we assume
\begin{equation}\label{5FGeq}
 m_*=\limsup_{s\to\infty}\|v(s)\|_{L^\infty(-1,1)}>0
\end{equation}
throughout this section.
Since $v$ satisfies $v_s={\cal A}v+K(y,s)v$ with $K(y,s)=-1+2u$,
this assumption is equivalent to $\limsup_{s\to\infty}\|v(s)\|_\rho>0$.
We apply Lemma \ref{5CLem} with $\delta=\bar{\delta}$ and $r=\bar{R}$.
Then there exists $\bar{m}\in(0,m_*)$ such that
if $\|v(s)\|_{L^\infty(-1,1)}<\bar{m}$, then it holds that
\[
 \sup_{-\bar{R}<y<\bar{R}}(|u(y,s)-1|+|u_y(y,s)|) < \bar{\delta}.
\]
From Lemma \ref{5DLem},
there exists $\{s_i\}_{i\in\N}$ ($s_i\to\infty$) such that
$\|v(s_i)\|_{L^\infty(-1,1)}\to0$ as $i\to\infty$.
By definition of $m_*$ ($\bar{m}<m_*$),
we can choose $s_i^-$ and $s_i^+$ ($s_i^-<s_i<s_i^+$) by
\[
\begin{array}{c}
 \|v(s)\|_{L^\infty(-1,1)} < \bar{m} \hfive \text{for } s\in(s_i^-,s_i^+),
\hfive\hfive
 \|v(s_i^{\pm})\|_{L^\infty(-1,1)} = \bar{m}.
\end{array}
\]
Since $\|v(s_i)\|_{L^\infty(-1,1)}\to0$ as $i\to\infty$,
we easily see that $\|v(s_i)\|_\rho+\|v_s(s_i)\|_\rho\to0$ as $i\to\infty$.
Therefor
it follow that $s_i^+-s_i\to\infty$ as $i\to\infty$.
Put
$\Delta_i = s_i^+-s_i^-$
$(\Delta_i\to\infty)$
and
\[
 u_i(y,s) = u(y,s_i^-+s),
\hfive
 v_i(y,s) = v(y,s_i^-+s).
\]
To analyze the dynamics of $v_i(s)$ in $L_\rho^2(\R)$,
we decompose a function $v_i$ by using eigenfunctions of ${\cal A}$.
\begin{equation}\label{5Geq}
 v_i
=
 \alpha_i\phi_0+\beta_i\phi_1+\gamma_i\phi_2+w_i,
\hfive\hfive
 \pa_yv_i
=
 \mu_i\phi_0+\nu_i\phi_1+q_i.
\end{equation}
%%%%%%%%%%%%%%%%%%%%%%%%%%%%%%%%%%%%%%%%%%%%%%%%%%%%%%%%%%%%%%%%

%%%%%%%%%%%%%%%%%%%%%%%%%%%%%%%  Lem  %%%%%%%%%%%%%%%%%%%%%%%%%%%%%%%
\begin{lem}\label{5ELem}
For any $d>0$,
it holds that
\[
 \liminf_{i\to\infty}\inf_{0<s<d}\|v_i(s)\|_\rho>0,
\hfive\hfive 
 \liminf_{i\to\infty}\inf_{0<s<d}(|\alpha_i(s)|+|\beta_i(s)|)>0.
\]
\end{lem}
%%%%%%%%%%%%%%%%%%%%%%%%%%%%%%%  Proof  %%%%%%%%%%%%%%%%%%%%%%%%%%%%%%%
\begin{proof}
First we assume
\[
 \liminf_{i\to\infty}\inf_{0<s<d}\|v_i(s)\|_\rho=0.
\]
Then
there exists $\{d_i\}_{i\in\N}\subset(0,d)$ such that
$\|v_i(d_i)\|_\rho\to0$ as $s\to\infty$.
By taking a subsequence,
we get $d_i\to d_*\in(0,d]$ and $(u_i,v_i)\to(U,V)$ as $i\to\infty$.
Then by definition of $s_i^-$ and $d_i$,
it follows that $V(0)\not\equiv0$ and $V(d_*)\equiv0$.
However
since $V$ satisfies $V_s={\cal A}v+(1-2U)V$,
$V(d_*)\equiv0$ contradicts the backward uniqueness for parabolic equations,
which proves the first statement.
To prove the second statement,
we repeat the same argument above.
Assume that there exists $\{d_i\}_{i\in\N}\in(0,d)$ such that
\begin{equation}\label{5GHeq}
 \liminf_{i\to\infty}(|\alpha_i(d_i)|+|\beta_i(d_i)|)=0.
\end{equation}
From the first statement of this lemma and Lemma \ref{5BLem},
we see that $|\Gamma(s_i)|<K$ for some $K>0$.
By taking a subsequence,
we get $d_i\to d_*$, $(u_i,v_i)\to(U,V)$ and $\Gamma(s_i)\to\Gamma_*\in(-K,K)$.
Then from definition of $\Gamma(s)$,
we see that
\[
 V(y,0)\leq0\hthree\text{for } y<\Gamma_*,
\hfive\hfive
 V(y,0)\geq0\hthree\text{for } y>\Gamma_*.
\]
Since $V\not\equiv0$ on $\R\times(0,\infty)$,
the number of zeros of $V(s)$ is decreasing in $s>0$.
Therefore the number of zeros of $V(d_*)$ is one or zero.
On the other hand,
we see from \eqref{5GHeq} that
$(V(d_*),\phi_0)_\rho=0$, $(V(d_*),\phi_1)_\rho=0$.
Therefore
from Corollary 6.17 \cite{Kotani},
we find that the number of $V(d_*)$ has more than one zeros,
which is contradiction.
The proof is completed.
\end{proof}
%%%%%%%%%%%%%%%%%%%%%%%%%%%%%%%%%%%%%%%%%%%%%%%%%%%%%%%%%%%%%%%%

%%%%%%%%%%%%%%%%%%%%%%%%%%%%%%%%%%%%%%%%%%%%%%%%%%%%%%%%%%%%%%%%
\subsubsection{Dynamics of $v_i(s)$ on $L_\rho^2(\R)$}\label{SubsubD}
In the following argument,
we always assume $s\in(0,\Delta_i)$.
Therefore
it follows from definition of $\bar{m}$ that
\[
 \sup_{-\bar{R}<y<\bar{R}}(|u_i(y,s)-1|+|\pa_yu_i(y,s)|)<\bar{\delta}
\hfive\text{for } s\in(0,\Delta_i).
\]
Then $v_i$ satisfies
\[
 \pa_sv_i = \pa_{yy}v_i-\frac{y}{2}\pa_yv_i+v_i+2(u_i-1)v_i.
\]
Multiplying equation by $\phi_k$ ($k=0,1,2$),
we get
\begin{equation}\label{5Heq}
 \dot{\alpha_i} = \alpha_i + 2h_{0i},
 \hfive
 \dot{\beta_i} = \frac{1}{2}\beta_i + 2h_{1i},
 \hfive
 \dot{\gamma_i} = 2h_{2i},
\end{equation}
where $h_{ki}$ ($k=0,1,2$) is given by
\[
 h_{ki} = \int_{-\infty}^\infty (u_i-1)v_i\phi_k\rho dy.
\]
Furthermore
since $w_i$ satisfies
\[
 \pa_sw_i = {\cal A}w_i+w_i+2(u_i-1)w_i+2(u_i-1)(\alpha_i\phi_0+\beta_i\phi_1+\gamma_i\phi_2)
 -2\sum_{k=0}^2h_{ki}\phi_k,
\]
we get
\begin{equation}\label{5Ieq}
 \frac{1}{2}\pa_s\|w_i\|_\rho^2
=
 -\|\pa_yw_i\|_\rho^2+\|w_i\|_\rho^2
+
 2\int_{-\infty}^\infty(u_i-1)w_i^2\rho dy
+
 2H_i,
\end{equation}
where $H_i$ is given by
\[
 H_i
=
 \int_{-\infty}^\infty
 (u_i-1)(\alpha_i\phi_0+\beta_i\phi_1+\gamma_i\phi_2)w_i\rho dy
 -
 \sum_{k=0}^2\int_{-\infty}^\infty h_{ki}\phi_kw_i\rho dy.
\]
By choice of $\bar{R}$ and $\bar{\delta}$,
we see that
\[
\begin{array}{c}
\dis
 \int_{-\infty}^\infty|u_i-1|w_i^2\rho dy
<
 \frac{\bar{\epsilon}}{8}\|w_i\|_{H_\rho^1(\R)}^2,
\hfive
 |h_{ki}|
< \left( \int_{-\infty}^\infty (u_i-1)^2\phi_k^2\rho dy\right)^{1/2}\|v_i\|_\rho
< \frac{\bar{\epsilon}}{24}\|v_i\|_\rho,
\\[4mm] \dis
\begin{array}{lll}
 |H_i|
\hspace{-2mm}&<&\hspace{-2mm} \dis
 \left(
 \int_{-\infty}^\infty
 (u_i-1)^2(|\phi_0|+|\phi_1|+|\phi_2|)^2\rho dy \right)^{1/2}\|v_i\|_\rho\|w_i\|_{\rho}
 +
 \|w_i\|_\rho\sum_{k=0}^2|h_{ki}|
\\ \dis
\hspace{-2mm}&<&\hspace{-2mm} \dis
 \frac{\bar{\epsilon}}{24}
 \|v_i\|_\rho\|w_i\|_{\rho}
 +
 \frac{\bar{\epsilon}}{8}\|v_i\|_\rho\|w_i\|_\rho
=
 \frac{\bar{\epsilon}}{6}\|v_i\|_\rho\|w_i\|_\rho.
\end{array}
\end{array}
\]
Applying these estimates in \eqref{5Heq} and \eqref{5Ieq},
we get
\begin{equation}\label{5Jeq}
 \begin{cases}
 \dis
 \pa_s\left( \alpha_i^2+\beta_i^2 \right) >
 \frac{1}{2}\left( \alpha_i^2+\beta_i^2 \right)-
 \bar{\epsilon}^2\left( \gamma_i^2+\|w_i\|_\rho^2 \right),
 \\[2mm] \dis
 \left| \pa_s\gamma_i^2 \right| <
 \frac{\bar{\epsilon}}{2}
 \left( (\alpha_i^2+\beta_i^2)+\gamma_i^2+\|w_i\|_\rho^2 \right),
 \\[1mm] \dis
 \pa_s\|w_i\|_\rho^2 <
 -\frac{1}{2}\|w_i\|_\rho^2+\bar{\epsilon}^2\left( (\alpha_i^2+\beta_i^2)+\gamma_i^2 \right).
 \end{cases}
\end{equation}
Next we provide estimates for $\pa_yv_i$.
Let $z_i=\pa_yv_i$.
Then $z_i$ satisfies
\[
 \pa_sz_i = {\cal A}z_i+\frac{z_i}{2}+2(u_i-1)z_i+2(\pa_yu_i)v_i.
\]
Since $z_i=\mu_i\phi_0+\nu_i\phi_1+q_i$,
$\mu_i$ and $\nu_i$ satisfy
\[
 \dot{\mu}_i = \frac{1}{2}\mu_i + 2\tilde{h}_{0i} - 2\hat{h}_{0i},
\hfive
 \dot{\nu}_i = 2\tilde{h}_{1i} - 2\hat{h}_{1i},
\]
where $\tilde{h}_{ki}$ and $\hat{h}_{ki}$ ($k=0,1$) are given by
\[
 \tilde{h}_{ki} = \int_{-\infty}^\infty (1-u_i)z_i\phi_k\rho dy,
\hfive
 \hat{h}_{ki} = \int_{-\infty}^\infty (\pa_yu_i)v_i\phi_k\rho dy.
\]
Furthermore
$q_i$ satisfies
\[
 \pa_sq_i =
 {\cal A}q_i+\frac{1}{2}q_i+2(u_i-1)q_i-2(\pa_yu_i)v_i+2(u_i-1)(\mu_i\phi_0+\nu_i\phi_1)-
 2(\tilde{h}_{0i}-\hat{h}_{0i})\phi_0+2(\tilde{h}_{1i}-\hat{h}_{1i})\phi_1.
\]
By the same calculation as $v_i$,
we obtain
\begin{equation}\label{5Keq}
 \begin{cases}
 \dis
 \pa_s \mu_i^2 >
 \frac{\mu_i^2}{2}-\bar{\epsilon}^2\left( \nu_i^2+\|q_i\|_\rho^2+\|v_i\|_\rho^2 \right),
 \\[3mm] \dis
 \left| \pa_s\nu_i^2 \right| <
 \frac{\bar{\epsilon}}{2}
 \left( \nu_i^2+\mu_i^2+\|q_i\|_\rho^2+\|v_i\|_\rho^2 \right),
 \\[1mm] \dis
 \pa_s\|w_i\|_\rho^2 <
 -\frac{1}{2}\|w_i\|_\rho^2+
 \bar{\epsilon}^2\left( \mu_i^2+\nu_i^2+\|v_i\|_\rho^2 \right).
 \end{cases}
\end{equation}
We here put
\begin{equation}\label{5KLeq}
 X_i = \alpha_i^2+\beta_i^2+\gamma_i^2,
\hfive
 Y_i = \mu_i^2+\nu_i^2,
\hfive
 Z_i = \|w_i\|_\rho^2+\|q_i\|_\rho^2.
\end{equation}
Since $\bar{\epsilon}<1/2$,
combining \eqref{5Jeq} and \eqref{5Keq},
we obtain
\begin{equation}\label{5Leq}
 \begin{cases}
 \dis
 \dot{X}_i > \frac{1}{4}X_i-\bar{\epsilon}(Y_i+Z_i),
 \\[1mm] \dis
 |\dot{Y}_i| < \bar{\epsilon}(X_i+Y_i+Z_i),
 \\[1mm] \dis
 \dot{Z}_i < -\frac{1}{4}Z_i+\bar{\epsilon}(X_i+Y_i).
 \end{cases}
\end{equation}
Let $\bar{\eta}>0$ be given in \eqref{5Feq}.
We define $\kappa_i$ by
\[
 \kappa_i = \bar{\eta}X_i-Y_i-Z_i.
\]
We investigate the behavior of $\kappa_i$.
\begin{align*}
 \kappa_i'
&>
 \frac{\bar{\eta}}{4}X_i-\bar{\eta}\bar{\epsilon}(Y_i+Z_i)-
 \bar{\epsilon}(X_i+Y_i+Z_i)+\frac{1}{4}Z_i-\bar{\epsilon}(X_i+Y_i)
\\
&=
 \left( \frac{\bar{\eta}}{4}-2\bar{\epsilon} \right)X_i-(2+\bar{\eta})\bar{\epsilon} Y_i
+
 \left( \frac{1}{4}-(1+\bar{\eta})\bar{\epsilon} \right)Z_i.
\end{align*}
Since $\kappa_i\geq0$ is equivalent to $Y_i+Z_i\leq\bar{\eta} X_i$,
it holds that
\begin{align*}
 \kappa_i'
&>
 \left( \frac{\bar{\eta}}{4}-2\bar{\epsilon}-(2+\bar{\eta})\bar{\epsilon}\bar{\eta} \right)X_i
+
 \left( \frac{1}{4}-(1+\bar{\eta})\bar{\epsilon} \right)Z_i
\\
&=
 \left(
 \left( \frac{1}{4}-2\bar{\epsilon} \right)\bar{\eta}-\left( 2+\bar{\eta}^2 \right)\bar{\epsilon}
 \right)X_i
+
 \left( \frac{1}{4}-(1+\bar{\eta})\bar{\epsilon} \right)Z_i
\hfive\text{if } \kappa_i>0.
\end{align*}
Therefore
from \eqref{5Feq},
we conclude
\[
 \kappa_i'>0
\hfive\text{if } \kappa_i\geq0.
\]
Since $Y_i=\gamma_i^2+\nu_i^2 = 2\gamma_i^2$ and
$Z_i=\|w_i\|_\rho^2+\|z_i\|_\rho^2=\|w_i\|_{H_\rho^1(\R)}^2$ (see Lemma 6.2 \cite{Filippas}),
if $\kappa_i<0$ ($\Leftrightarrow\bar{\eta}X_i<Y_i+Z_i$) and $Z_i < \bar{\zeta}Y_i$,
it holds that
\begin{align*}
 \alpha_i^2+\beta_i^2+\|w_i\|_{H_\rho^1}^2
<
 X_i+Z_i
<
 \left( \frac{1+\bar{\zeta}}{\bar{\eta}}+\bar{\zeta} \right)Y_i
=
 2\left( \frac{1+\bar{\zeta}}{\bar{\eta}}+\bar{\zeta} \right)\gamma_i^2
<
 \epsilon_1\gamma_1^2,
\end{align*}
where we use \eqref{5Feq} in the last inequality.
Therefore
by definition of $\epsilon_1$,
$v_i$ has more than one zeros if $\kappa_i<0$ and $Z_i<\bar{\zeta}Y_i$.
Summarizing the above estimates,
we obtain the following lemma.
%%%%%%%%%%%%%%%%%%%%%%%%%%%%%%%%%%%%%%%%%%%%%%%%%%%%%%%%%%%%%%%%

%%%%%%%%%%%%%%%%%%%%%%%%%%%%%%%  Lem  %%%%%%%%%%%%%%%%%%%%%%%%%%%%%%%
\begin{lem}\label{5FLem}
If $\kappa_i(s')\geq0$ for some $s'\in(0,\Delta_i)$,
then it holds that $\kappa_i(s)>0$ for $s\in(s',\Delta_i)$.
Furthermore
if $\kappa_i(s)<0$ for some $s\in(0,\Delta_i)$,
then it holds that $\bar{\zeta}Y_i(s)<Z_i(s)$.
\end{lem}
%%%%%%%%%%%%%%%%%%%%%%%%%%%%%%%%%%%%%%%%%%%%%%%%%%%%%%%%%%%%%%%%

%%%%%%%%%%%%%%%%%%%%%%%%%%%%%%%  Lem  %%%%%%%%%%%%%%%%%%%%%%%%%%%%%%%
\begin{lem}\label{5GLem}
Let $\Delta_i^-=\{s\in(0,\Delta_i); \kappa_i(s')<0$ {\rm for} $s'\in(0,s)\}$.
Then it holds that $\dis\lim_{i\to\infty}\Delta_i^-=\infty$ and
$\bar{\zeta}Y_i(s)<Z_i(s)$ for $s\in(0,\Delta_i^-)$. 
\end{lem}
%%%%%%%%%%%%%%%%%%%%%%%%%%%%%%%  Proof  %%%%%%%%%%%%%%%%%%%%%%%%%%%%%%%
\begin{proof}
Since the second statement is trivial from Lemma \ref{5FLem},
it is enough to  show the first statement.
We prove by contradiction.
Assume that there exists a subsequence $\{j\}_{j\in\Lambda}\subset\{i\}_{i\in\N}$ such that
$\{\Delta_{j}^{-}\}_{j\in\Lambda}$ is bounded.
Then from Lemma \ref{5ELem},
there exists $\theta>0$ such that
\begin{equation}\label{5Neq}
 \inf_{0<s<\Delta_j^-}(|\alpha_j(s)|+|\beta_j(s)|) > \theta
\hfive\text{for } j\in\Lambda.
\end{equation}
From definition of $\Delta_i^-$ and Lemma \ref{5FLem},
we see that $\kappa_i(s)>0$ for $s\in(\Delta_i^-,\Delta_i)$.
Therefore
since $X_i$, $Y_i$ and $Z_i$ satisfy \eqref{5Leq} for $s\in(0,\Delta_i)$,
we get from \eqref{5Feq} that
\begin{align*}
 \dot{X}_i & > \frac{1}{4}X_i-\bar{\epsilon}\bar{\eta}X_i > \frac{1}{8}X_i
\hfive\text{for } s\in(\Delta_i^-,\Delta_i).
\end{align*}
Since we note from \eqref{5Neq} that $X_j(\Delta_j^-)>\theta$ for $j\in\Lambda$,
we obtain
\[
 X_j(s) > \theta e^{(s-\Delta_j^-)/8}
\hfive\text{for} s\in(\Delta_j^-,\Delta_j).
\]
However
since $\Delta_j\to\infty$ as $j\to\infty$ and $\Delta_j^-$ is bounded,
$X_j(s)$ becomes arbitrary large for large $j\in\Lambda$, 
which contradicts a boundedness of $X_i(s)$.
\end{proof}
%%%%%%%%%%%%%%%%%%%%%%%%%%%%%%%%%%%%%%%%%%%%%%%%%%%%%%%%%%%%%%%%

%%%%%%%%%%%%%%%%%%%%%%%%%%%%%%%  Proof  %%%%%%%%%%%%%%%%%%%%%%%%%%%%%%%
\begin{proof}[{\bf Proof of Proposition \ref{5BPro}}]\label{ProofPro}
From Lemma \ref{5GLem},
there existss a subsequence $\{(u_,v_i)\}_{i\in\N}$ such that
\begin{equation}\label{5Oeq}
 \bar{\eta}X_i < Y_i+Z_i,
 \hthree
 \bar{\zeta}Y_i<Z_i
\hfive\text{for } s\in(0,\Delta_i^-),
\hfive\hfive
 \lim_{i\to\infty}\Delta_i^-=\infty.
\end{equation}
Therefore
we get from \eqref{5Leq} that
\[
 \dot{Z}_i <
 \left(
 -\frac{1}{4}+\bar{\epsilon}\left( \frac{1}{\bar{\eta}}\left(1+\frac{1}{\bar{\zeta}}
 \right)
 +\frac{1}{\bar{\zeta}}\right) \right)Z_i
<
 -\frac{1}{8}Z_i
\hfive\text{for } s\in(0,\Delta_i^-),
\]
which implies $Z_i < Z_i(0)e^{-s/8}$.
Combining this estimate and \eqref{5Oeq},
we obtain
\[
 \|v_i(s)\|_\rho < ce^{-s/8}
\hfive\text{for } s\in(0,\Delta_i^-)
\]
for some $c>0$.
As as consequence,
from Lemma \ref{5CLem},
there exists a positive continuous function $F(s)$ on $s>0$ such that
$F(s)\to0$ as $s\to\infty$ and
\[
 \|u_i(s)-1\|_\rho < F(s) \hfive\text{for } s\in(0,\Delta_i^-).
\]
Then by taking a subsequence,
we get $(u_i,v_i)\to(U,V)$ as $i\to\infty$.
From above estimates,
we see that
\[
 \lim_{s\to\infty}\|U(s)-1\|_\rho=0,
\hfive\hfive
 \|V(s)\|_\rho=O(e^{-s/8}).
\]
Then
Lemma \ref{5ALem} implies that
$\|U(s)-1\|_\rho=O(e^{-\gamma s})$ for some $\gamma>0$.
Therefore
we get form Lemma \ref{5HLem} that
\begin{equation}\label{5Peq}
 |U(y,s)-1|+|V(y,s)|<ce^{-\gamma s/2} \hfive\text{for } |y|<e^{\theta s}
\end{equation}
for some $\theta>0$ and $c>0$.
Since $V$ satisfies \eqref{eq2B},
it holds that
\[
 \|V_s-({\cal A}-1)V\|_\rho < 2\|(U-1)V\|_\rho.
\]
Since $U$ is uniformly bounded,
by using \eqref{5FFGeq} and \eqref{5Peq},
we get
\begin{align*}
 \|(U-1)V\|_\rho^2
&=
 \int_{|y|<e^{\theta s}}(U-1)^2V^2\rho dy
+
 \int_{|y|>e^{\kappa s}}(U-1)^2V^2\rho dy
\\
&<
 ce^{-2\gamma_1 s}\int_{\R}V^2\rho dy
+
 ce^{-2\theta s}\int_{|y|>e^{\kappa s}}|y|^2V^2\rho dy
<
 c\left( e^{-2\gamma_1 s}+e^{-2\theta s} \right)\|V\|_{H_\rho^1(\R)}^2.
\end{align*}
Therefore
we obtain
\[
 \|V_s-({\cal A}-1)V\|_\rho < ce^{-\mu s}\|V\|_{H_\rho^1(\R)}^2
\]
for some $\mu>0$.
Repeating the argument as in the proof of Proposition \ref{5APro},
which derives contradiction.
Therefore the assumption \eqref{5FGeq} is false.
\end{proof}
%%%%%%%%%%%%%%%%%%%%%%%%%%%%%%%%%%%%%%%%%%%%%%%%%%%%%%%%%%%%%%%%

%%%%%%%%%%%%%%%%%%%%%%%%%%%%%%%  Lem  %%%%%%%%%%%%%%%%%%%%%%%%%%%%%%%
\begin{lem}\label{5ALem}
If $(u_i,v_i)$ converges to some function $(U,V)$
in $L_{\mathrm{loc}}^\infty(\R\times(0,\infty))$ satisfying
$\dis\lim_{s\to\infty}\|U(s)-1\|_\rho=0$ and $\|V(s)\|_\rho$ decays exponentially,
then $\|U(s)-1\|_\rho$ decays exponentially.
\end{lem}
%%%%%%%%%%%%%%%%%%%%%%%%%%%%%%%  Proof  %%%%%%%%%%%%%%%%%%%%%%%%%%%%%%%
\begin{proof}
Put $\lambda_i=e^{-s_i}$, $t_i=T-\lambda_i$ and
\[
 a_i(x,\tau) = \lambda_ia(\sqrt{\lambda_i}x,t_i+\lambda_i\tau),
\hfive
 b_i(x,\tau) = \lambda_ib(\sqrt{\lambda_i}x,t_i+\lambda_i\tau).
\]
Then we see that
\[
 u_i(y,s)=(1-\tau)a_i(x,\tau),
\hfive
 v_i(y,s)=(1-\tau)b_i(x,\tau)
\]
with $x=e^{-s/2}y$ and $1-\tau=e^{-s}$.
Therefore
since $(a_i(0),b_i(0))=(u_i(0),v_i(0))$,
we get $(a_i,b_i)\to(A,B)$ and
\[
 U(y,s) = (1-\tau)A(x,\tau),
\hfive
 V(y,s) = (1-\tau)B(x,\tau).
\]
Since $\|V(s)\|_\rho=O(e^{-\gamma s})$,
we see from Lemma \ref{5HLem} that $|V(y,s)|=O(e^{-\gamma_1s})$ for $|y|<e^{\theta s}$.
Therefore
applying the same argument as \cite{Filippas} with a slight modification,
we find that there are two possibilities:
(I) there exists $\gamma_1>0$ such that $\|U(s)-1\|_\rho=O(e^{-\gamma_1s})$
or
(II) there exists $\Lambda\not=0$ such that $U(s)-1=\Lambda(1+o(1))s^{-1}\phi_2$ in $L_\rho^2(\R)$.
Assume that (II) holds.
Since $|V(y,s)|=O(e^{-\gamma_1s})$ for $|y|<e^{\theta s}$,
the argument in the proof of Proposition 2.3 \cite{Herrero**} shows
\[
 \lim_{s\to\infty}\sup_{|y|<l\sqrt{s}}
 \left| U(y,s)-\left( 1+\frac{\bar{c}}{s}y^2 \right)^{-1} \right|
 =0
\hfive\text{for any } l>0
\]
with some $\bar{c}>0$.
Furthermore applying the argument in \cite{Herrero},
we can verify that the origin is an isolated blow-up point of $(A,B)$.
Therefore (II) is excluded from Lemma \ref{3BCLem},
which completes the proof.
\end{proof}
%%%%%%%%%%%%%%%%%%%%%%%%%%%%%%%%%%%%%%%%%%%%%%%%%%%%%%%%%%%%%%%%

%%%%%%%%%%%%%%%%%%%%%%%%%%%%%%%  Lem  %%%%%%%%%%%%%%%%%%%%%%%%%%%%%%%
\begin{lem}\label{5HLem}
Let $(U,V)$ be a sbounded solution of \eqref{eq2B} satisfying $\|V(s)\|_\rho=O(e^{-\gamma s})$.
Then there exist $\theta>0$ and $c>0$ such that
\[
 |V(y,s)|<ce^{-\gamma s/2}
\hfive\mathrm{for}\ |y|<e^{\theta s}.
\]
Furthermore
if $\|U(s)-1\|_\rho+\|V(s)\|_\rho=O(e^{-\gamma s})$ .
Then there exist $\theta>0$ and $c>0$ such that
\[
 |U(y,s)-1|+|V(y,s)|<ce^{-\gamma s/2}
\hfive\mathrm{for}\ |y|<e^{\theta s}.
\]
\end{lem}
%%%%%%%%%%%%%%%%%%%%%%%%%%%%%%%  Proof  %%%%%%%%%%%%%%%%%%%%%%%%%%%%%%%
\begin{proof}
We apply the method given in \cite{Herrero}.
Let $K=2\sup_{s>0}\|U(s)\|_{L^\infty(\R)}$.
To construct a comparison function for $V$,
we consider
\[
 W_\tau = {\cal A}W+KW \hthree \tau>s,
\hfive
 W_0=|V(s)|.
\]
Then this solution $W$ is given by
\[
 W(\tau) =
 \frac{e^{K(\tau-s)}}{2\sqrt{\pi}\sqrt{1-e^{-(\tau-s)}}}
 \int_{-\infty}^\infty
 \exp\left( -\frac{(ye^{-(\tau-s)/2}-\xi)^2}{4(1-e^{-(\tau-s)})} \right)
 W_0(\xi)d\xi.
\]
Then
it holds that
\[
 \int_{-\infty}^\infty
 \exp\left( -\frac{(ye^{-(\tau-s)/2}-\xi)^2}{4(1-e^{-(\tau-s)})} \right)
 W_0(\xi)d\xi
<
 \left( 
 \int_{-\infty}^\infty
 \exp\left( -\frac{(ye^{-(\tau-s)/2}-\xi)^2}{2(1-e^{-(\tau-s)})}+\frac{\xi^2}{4} \right)^{1/2}
 \right)
 \|W_0\|_\rho.
\]
Since
\[
 -\frac{(ye^{-(\tau-s)/2}-\xi)^2}{2(1-e^{-(\tau-s)})}+\frac{\xi^2}{4}
=
 -\frac{1+e^{-(\tau-s)}}{4(1-e^{-(\tau-s)})}
 \left( \xi-\frac{2e^{-(\tau-s)/2}}{1+e^{-(\tau-s)}}y \right)^2
+
 \frac{e^{-(\tau-s)}}{4(1-e^{-(\tau-s)})^2}\left(2-e^{-(\tau-s)}\right)y^2,
\]
we obtain
\[
 W(\tau) < c\left( \frac{1+e^{-(\tau-s)}}{2\pi(1-e^{-(\tau-s))})} \right)^{1/4}e^{K(\tau-s)}
 \exp\left( \frac{e^{-(\tau-s)}y^2}{2(1-e^{-(\tau-s)})^2} \right)
 e^{-\gamma s}.
\]
We choose $\tau=(1+\frac{\gamma}{2K})s$.
Since $\tau-s=\frac{\gamma s}{2K}>\log2$ for large $s>0$,
it follows that
\[
 W(\tau) < ce^{-\gamma s/2}\exp\left( 2e^{-\gamma s/2k}y^2 \right) < ce^{-\gamma s/2}
\]
for $|y|<e^{\gamma s/4K}$ and $s\gg1$.
Therefore
the first estimate is proved.
Next we provide estimates for $U-1$.
Let $C=U-1$.
Then it satisfies
\[
 C_s = {\cal A}C+C+C^2-V^2.
\]
By the same way as above,
we consider
\[
 W_\tau={\cal A}W+KW+V^2 \hthree \tau>s,
\hfive |W_0|=|V(s)|,
\]
where $K=1+\sup_{s>0}\|U(s)\|_{L^\infty(\R)}$.
Then
$W$ is given by
\begin{align*}
 W(\tau)
&=
 \frac{e^{K(\tau-s)}}{2\sqrt{\pi}\sqrt{1-e^{-(\tau-s)}}}
 \int_{-\infty}^\infty
 \exp\left( -\frac{(ye^{-(\tau-s)/2}-\xi)^2}{4(1-e^{-(\tau-s)})} \right)
 W_0(\xi)d\xi
\\ & \hfive\hfive
+
 \int_s^\tau
 \frac{e^{K(\tau-\mu)}}{2\sqrt{\pi}\sqrt{1-e^{-(\tau-\mu)}}}d\mu
 \int_{-\infty}^\infty
 \exp\left( -\frac{(ye^{-(\tau-\mu)/2}-\xi)^2}{4(1-e^{-(\tau-\mu)})} \right)
 V(\xi,\mu)^2d\xi.
\end{align*}
By the same way as above,
we choose $\tau=(1+\frac{\gamma}{2K})s$.
Then
it is enough to estimate the second term on the right-hand side.
Since $|V(\xi,s)|<ce^{-\gamma s/2}$ for $|\xi|<e^{\theta s}$ and $s\gg1$,
we get
\[
\begin{array}{l}
\dis
 \int_{-\infty}^\infty
 \exp\left( -\frac{(ye^{-(\tau-\mu)/2}-\xi)^2}{4(1-e^{-(\tau-\mu)})} \right)
 V(\xi,\mu)^2d\xi
<
 \int_{|\xi|<e^{\theta s}}d\xi+\int_{|\xi|>e^{\theta s}}d\xi
\\[4mm] \dis \hfive\hfive
<
 ce^{-\gamma \mu}\sqrt{4\pi(1-e^{-\tau-\mu})}
+
 c\int_{|\xi|>e^{\theta s}}
 \exp\left( -\frac{(ye^{-(\tau-\mu)/2}-\xi)^2}{4(1-e^{-(\tau-\mu)})} \right)
 d\xi.
\end{array}
\]
If $|y|<e^{\theta s}/2$ and $|\xi|>e^{\theta s}$,
it holds that $|ye^{-(\tau-\mu)/2}-\xi|>|\xi|/2$.
Therefore
we get
\[
\begin{array}{l}
\dis
 \int_{|\xi|>e^{\theta s}}
 \exp\left( -\frac{(ye^{-(\tau-\mu)/2}-\xi)^2}{4(1-e^{-(\tau-\mu)})} \right)
 d\xi
<
 \int_{|\xi|>e^{\theta s}}
 \exp\left( -\frac{\xi^2}{16(1-e^{-(\tau-\mu)})} \right)
 d\xi
\\[4mm] \dis \hfive\hfive
<
 e^{-\gamma s}
 \int_{|\xi|>e^{\theta s}}
 |\xi|^{\gamma/\theta}\exp\left( -\frac{\xi^2}{16(1-e^{-(\tau-\mu)})} \right)
<
 c(1-e^{-(\tau-\mu)})^{(\gamma+\theta)/2\theta}e^{-\gamma s}
\end{array}
\]
for $|y|<e^{\theta s}/2$.
As a consequence,
we obtain
\[
\begin{array}{l}
\dis
 \int_s^\tau
 \frac{e^{K(\tau-\mu)}}{2\sqrt{\pi}\sqrt{1-e^{-(\tau-\mu)}}}d\mu
 \int_{-\infty}^\infty
 \exp\left( -\frac{(ye^{-(\tau-\mu)/2}-\xi)^2}{4(1-e^{-(\tau-\mu)})} \right)
 V(\xi,\mu)^2d\xi
\\[4mm] \dis \hfive\hfive
<
 c\int_s^\tau e^{K(\tau-\mu)}e^{-\gamma \mu}d\mu
<
 ce^{K(\tau-s)}\int_s^\tau e^{-\gamma \mu}d\mu
<
 ce^{K(\tau-s)}e^{-\gamma s}
=
 ce^{-\gamma s/2}
\end{array}
\]
for $|y|<e^{\theta s}/2$,
which completes the proof.
\end{proof}
%%%%%%%%%%%%%%%%%%%%%%%%%%%%%%%%%%%%%%%%%%%%%%%%%%%%%%%%%%%%%%%%

%%%%%%%%%%%%%%%%%%%%%%%%%%%%%%%%%%%%%%%%%%%%%%%%%%%%%%%%%%%%%%%%
\subsection{Proof of Theorem \ref{thm3}}
%%%%%%%%%%%%%%%%%%%%%%%%%%%%%%%  Proof  %%%%%%%%%%%%%%%%%%%%%%%%%%%%%%%
The proof of Theorem \ref{thm3} is almost the same as that of Proposition \ref{5BPro}.
\begin{proof}[Proof of Theorem {\rm\ref{thm3}}]
Assume that \eqref{5Beq} holds true.
Then from Proposition \ref{5BPro},
$v(s)$ converges to zero in $L_\rho^2(\R)$ as $s\to\infty$.
Then we see from Lemma \ref{5CLem} that $u(s)\to1$ in $L_\rho^2(\R)$ as $s\to\infty$.
Once $\|v(s)\|_\rho=O(e^{-\gamma s})$ for some $\gamma>0$ is derived,
by the same argument as in the proof of Proposition \ref{5BPro},
we obtain contradiction.
Therefore
it is enough to show that $v(s)$ decays exponentially in $L_\rho^2(\R)$.
In fact,
we decompose $v(s)$ as \eqref{5Geq} and define $X$, $Y$ and $Z$ as \eqref{5KLeq}.
Since $(u(s),v(s))\to(0,1)$,
repeating arguments in Section \ref{SubsubD},
we obtain \eqref{5Leq}.
Therefore
we obtain
\[
 \kappa(s)=\bar{\eta}X(s)-Y(s)-Z(s)<0,
\hfive
 \bar{\zeta}Y(s)<Z(s).
\]
This implies that $v(s)$ decays exponentially in $L_\rho^2(\R)$,
which completes the proof.
\end{proof}
%%%%%%%%%%%%%%%%%%%%%%%%%%%%%%%%%%%%%%%%%%%%%%%%%%%%%%%%%%%%%%%%

%%%%%%%%%%%%%%%%%%%%%%%%%%%%%%%%%%%%%%%%%%%%%%%%%%%%%%%%%%%%%%%%
\section*{Acknowledgement}
The author is partly supported by
Grant-in-Aid for Young Scientists (B) No. 26800065.
%%%%%%%%%%%%%%%%%%%%%%%%%%%%%%%%%%%%%%%%%%%%%%%%%%%%%%%%%%%%%%%%

%%%%%%%%%%%%%%%%%%%%%%%%%%%%%%%%%%%%%%%%%%%%%%%%%%%%%%%%%%%%%%%%

%%%%%%%%%%%%%%%%%%%%%%%%%%%%%%%%%%%%%%%%%%%%%%%%%%%%%%%%%%%%%%%%

\begin{thebibliography}{99}
\bibitem{Ackermann}
 N. Ackermann, T. Bartsch,
 Superstable manifolds of semilinear parabolic problems,
 J. Dynam. Differential Equations {\bf 17} no. 1 (2005) 115-173.

\bibitem{Cohen}
 P. J. Cohen, M. Lees,
 Asymptotic decay of solutions of differential inequalities,
 Pacific J. Math. {\bf 11} (1961) 1235-1249.

\bibitem{Constantin}
 P. Constantin, P. D. Lax, and A. Majda
 A simple one-dimensional model for the three-dimensional vorticity equation,
 Comm. Pure Appl. Math. {\bf 38} no. 6 (1985) 715-724.

\bibitem{Filippas}
 S. Filippas, R. V. Cohn,
 Refined asymptotics for the blowup of $u_t-\Delta u=u^p$,
 Comm. Pure Appl. Math. {\bf 45} no. 7 (1992) 821-869.

\bibitem{Fujita}
 H. Fujita,
 On the blowing up of solutions of the Cauchy problem for $u_t=\Delta u+u^{1+\alpha}$,
 J. Fac. Sci. Univ. Tokyo Sect. I {\bf 13} (1966) 109-124 
 
\bibitem{Guo}
 J. S. Guo, H. Ninomiya, M. Shimojo, E. Yanagida,
 Convergence and blow-up of solutions for a complex-valued heat equation with a quadratic nonlinearity,
 Trans. Amer. Math. Soc. {\bf 365} no. 5 (2013) 2447-2467.

% \bibitem{Harada}
% J. Harada

\bibitem{Herrero}
 M. A. Herrero, J. J. L. Vel\'azquez,
 Blow-up profiles in one-dimensional, semilinear parabolic problems,
 Comm. Partial Differential Equations
 {\bf17} no. 1-2 (1992) 205-219

% \bibitem{Herrro*}
%  M. A. Herrero, J. J. L. Vel\'azquez,
%  Flat blow-up in one-dimensional semilinear heat equations
%  Differential Integral Equations {\bf 5} no. 5 (1992) 973-997
 
\bibitem{Herrero**}
 J. J. L. Vel\'azquez,
 Higher-dimensional blow up for semilinear parabolic equations,
 Comm. Partial Differential Equations
 {\bf 17} no. 9-10 (1992) 1567-1596
 
% \bibitem{Herrero2}
%  M. A. Herrero, J. J. L. Vel\'azquez,
%  Blow-up behaviour of one-dimensional semilinear parabolic equations,
%  Ann. Inst. H. Poincar\'e Anal. Non Lin\'eaire {\bf 10} no. 2 (1993) 131-189

\bibitem{Kotani}
 S. Kotani, H. Matano,
 Differential Equation and Eigenvalue problems,
 Iwanami-shoten (Japanese)

\bibitem{Matano}
 H. Matano,
 Nonincrease of the lap-number of a solution for a one-dimensional semilinear parabolic equation,
 J. Fac. Sci. Univ. Tokyo Sect. IA Math. {\bf 29} no. 2 (1982) 401-441

\bibitem{Naito}
 Y. Naito, T. Suzuki,
 Existence of type II blowup solutions for a semilinear heat equation with critical nonlinearity,
 J. Differential Equations {\bf 232} no. 1 (2007) 176-211

\bibitem{Ogawa}
 H. Ogawa,
 Lower Bounds for Solutions of Differential Inequalities in Hilbert Space,
 Proc. Amer. Math. Soc. {\bf 16} (1965) 1241-1243.

\bibitem{Palais}
 R. Palais
 Blowup for nonlinear equations using a comparison principle in Fourier space,
 Comm. Pure Appl. Math. {\bf 41} no. 2 (1988) 165-196.

\bibitem{Sakajo1}
 T. Sakajo,
 Blow-up solutions of the Constantin-Lax-Majda equation with a generalized viscosity term,
 J. Math. Sci. Univ. Tokyo {\bf 10} no. 1 (2003) 187-207.

\bibitem{Sakajo2}
 T. Sakajo,
 On global solutions for the Constantin-Lax-Majda equation with a generalized viscosity term,
 Nonlinearity {\bf 16} no. 4 (2003) 1319-228.

\bibitem{Schochet}
 S. Schochet,
 Explicit solutions of the viscous model vorticity equation,
 Comm. Pure Appl. Math. {\bf 39} no. 4 (1986) 531-537.

\bibitem{Yang}
 Y. Yang,
 Behavior of solutions of model equations for incompressible fluid flow,
 J. Differential Equations {\bf 125} no. 1 (1996) 133-153.
 
\bibitem{Nouaili}
 N. Nouaili, H. Zaag,
 Profile for a simultaneously blowing up solution for a complex valued semilinear heat equation,
 arXiv: 1306.4435.
\end{thebibliography}
\end{document}